\documentclass[12pt]{amsart}
\usepackage{amsmath,amssymb,amsbsy,amsfonts,amsthm,latexsym,
            amsopn,amstext,amsxtra,euscript,amscd,amsthm,graphicx,comment}
\usepackage[left=3.5cm,right=3.5cm,top=2cm,bottom=2cm,includeheadfoot]{geometry}
\usepackage[enableskew]{youngtab} 
\usepackage{caption}

\newtheorem{prop}{Proposition}
\newtheorem{proposition}[prop]{Proposition}

\newtheorem{thm}{Theorem}
\newtheorem{theorem}[thm]{Theorem}

\newtheorem{cor}{Corollary}
\newtheorem{corollary}[cor]{Corollary}

\theoremstyle{remark}
\newtheorem*{remark}{\bf{Remarks}}

\theoremstyle{remar}

{Example}

\newtheorem*{Acknowledgement}{Acknowledgements}

\def\\{\cr}
\def\({\left(}
\def\){\right)}
\def\<{\langle}
\def\>{\rangle}

\def\func#1{\mathop{\rm #1}}%

\begin{document}
\title[Growth of polynomials]{On the growth and zeros
of polynomials attached to arithmetic functions}
\author{Bernhard Heim }
\address{Faculty of Mathematics, Computer Science, and Natural Sciences,
RWTH Aachen University, 52056 Aachen, Germany}
\email{bernhard.heim@rwth-aachen.de}
\author{Markus Neuhauser}
\address{Kutaisi International University (KIU), 5/7, Youth Avenue,  Kutaisi, 4600, Georgia}
\email{markus.neuhauser@kiu.edu.ge}

\pagenumbering{arabic}

\begin{abstract}
In this paper we investigate growth properties and the zero distribution of
polynomials attached to arithmetic functions $g$ and $h$, where $g$ is
normalized, of moderate growth, and
$0<h(n) \leq h(n+1)$. We put $P_0^{g,h}(x)=1$ and 
\begin{equation*}
P_n^{g,h}(x) := \frac{x}{h(n)} \sum_{k=1}^{n} g(k) \, P_{n-k}^{g,h}(x).
\end{equation*}
As an application we obtain the best known result on the domain of the non-vanishing 
of the Fourier coefficients of powers of the Dedekind $\eta$-function. Here, $g$ is the sum of divisors and $h$ the identity function.
Kostant's result on the representation of 
simple complex Lie algebras and Han's  results on the Nekrasov--Okounkov hook length formula are extended. 
The polynomials are related to reciprocals of Eisenstein series, Klein's $j$-invariant, and Chebyshev polynomials of the second kind. 
\end{abstract}

\maketitle
\newpage
\section{Introduction}
Properties of coefficients of generating series \cite{Wi06}, especially
Fourier coefficients of powers of the Dedekind $\eta$-function have been 
the focus of research since the times of Euler \cite{Ma72,Se85, Ap90, AE04, On03,HNW18}:
\begin{equation}
\label{eq:1}
\eta \left( \tau \right) ^{r}:=q^{\frac{r}{24}}\prod _{m=1}^{\infty }\left( 1-q^{m}\right) ^{r}=
q^{\frac{r}{24}}\sum _{n=0}^{\infty }a_{n}\left( r \right) q^{n}.
\end{equation}
Here, $q := e^{2\pi i\tau }$, $\func{Im}\left( \tau \right) >0$ and $r \in \mathbb{Z}$. 
The coefficients are special values of the D'Arcais polynomials $P_n(x)$ \cite{DA13, Ne55, Co74, We06}.
It has been recently noticed that the growth and vanishing properties of these polynomials
have much in common with properties of other interesting polynomials \cite{HLN19,HN20B}.
These include special orthogonal polynomials as associated Laguerre 
polynomials and Chebyshev polynomials of the second kind. Also included are polynomials 
attached to reciprocals of the Klein's $j$-invariant and Eisenstein series \cite{HN20A,HN20C}.

In this paper we investigate growth properties and the zero distribution of
polynomials attached to arithmetic functions $g$ and $h$ inspired by Rota \cite{KRY09}. 

Let $g$ be normalized and of moderate growth.
Further, let $0<h(n) \leq h(n+1)$. We put $P_0^{g,h}(x)=1$ and 
\begin{equation} \label{gh}
P_n^{g,h}(x) := \frac{x}{h(n)} \sum_{k=1}^{n} g(k) \, P_{n-k}^{g,h}(x).
\end{equation}
This definition includes all mentioned examples. 
Before providing examples
and
explicit formulas for these polynomials, we give one application for the coefficients of the
Dedekind $\eta $-function. Let $g(n)=\sigma (n):= \sum_{d \mid
n}d$, $h(n)= \func{id}(n)=n$ and 
$a_n(r)$ be defined by (\ref{eq:1}), the $n$th coefficient of the $r$th power of the
Dedekind $\eta $-function. Han \cite{Ha10} observed that the Nekrasov--Okounkov hook length formula 
\cite{NO06, We06} implies that $a_n(r) \neq 0$ if $ r > n^2 -1$. This improves previous results by Kostant \cite{Ko04}.
In \cite{HN20B} we proved that
\begin{equation} \label{improve}
a_n(r) \neq 0 \text{ holds for }  r > \kappa \cdot
(n-1) \text{ where } \kappa =15. 
\end{equation}
Numerical investigations show that $\kappa$ has to be larger than $9.55$ (see Table \ref{roots}). 
In this paper we prove that (\ref{improve}) is already true for $\kappa = 10.82$.

Since the definition of $P_n^{g,h}(x)$ is quite
abstract, we provide two examples of families of polynomials,
to familiarize the reader with the types of polynomials
we are studying. At first, they appear
to have nothing in common.

Let us start with the Nekrasov--Okounkov hook length formula \cite{NO06}.
Let $\eta(\tau)$ be the Dedekind $\eta$-function.
Let $\lambda$ be a partition of $n$ 
and let $|\lambda|=n$. By
$\mathcal{H}(\lambda)$ we denote the multiset of hook lengths associated
with $\lambda$
and by $\mathcal{P}$,
the set of all partitions.
The Nekrasov--Okounkov hook length formula (\cite{Ha10}, Theorem 1.2) states that
\begin{equation} \label{ON}
\sum_{n=0}^{\infty} P_n^{\sigma}(z) \, q^n =
\sum_{ \lambda \in \mathcal{P}} q^{|\lambda|} \prod_{ h \in \mathcal{H}(\lambda)} \left(  1 + \frac{z-1}{h^2} \right) 
=   
q^{\frac{z}{24}} \eta(\tau)^{-z}.
\end{equation}
The identity (\ref{ON}) is valid for all $z \in \mathbb{C}$.
Note that the $P_n^{\sigma}(x)$ are integer-valued polynomials of degree $n$. 
From the formula it follows that
$(-1)^n P_n^{\sigma}(x) >0$ for all real $x < -(n^2+1)$.

The second example is of a more artificial nature, discovered recently \cite{HN20A}, when
studying the $q$-expansion of the reciprocals of Klein's $j$-invariant and reciprocals of Eisenstein series \cite{BB05,BK17,HN20C}.
Let $$j(\tau)= \sum_{n=-1}^{\infty} c(n) q^n = q^{-1} + 744 + 196884q+ \ldots$$ 
denote Klein's $j$-invariant.
Asai, Kaneko, and Ninomiya \cite{AKN97} proved
that the coefficients of the $q$-expansion of
$1/j(\tau)$ are non-vanishing and have strictly alternating signs. This follows from
their result on the zero distribution of the $n$th
Faber polynomials $\varphi_n \left( x\right) $ and
the denominator formula for the monster Lie algebra. The zeros of the Faber polynomials are simple 
and lie in the interval $(0,1728)$. They obtained the remarkable identity:
\begin{equation}
\frac{1}{j(\tau)} = \sum_{n=1}^{\infty} \varphi_n'(0) \, 
\frac{q^n}{n}.\end{equation}
Let $c^{*}(n):= c(n)/744$. Define the polynomials $Q_{j,n}(x)$ by
\begin{equation}
\sum_{n=0}^{\infty} Q_{j,n}(x) \, q^n :=
\frac{1}{ 1 - x \sum_{n=1}^{\infty} c^{*}(n) \, q^n }.
\end{equation}
We have proved in \cite{HN20A} that $Q_{j,n}(x) = Q_{\gamma_2,n}(x) + 2x Q_{\gamma_2,n}'(x)+ \frac{x^2}{2}  Q_{\gamma_2,n}''(x)$, where $Q_{\gamma_2,n}(x)$ are polynomials attached to Weber's cubic root
function $\gamma_2$ of $j$ in a similar way. We have also proved that $Q_{\gamma_2,n}(z)\neq 0$ for all $ \vert z \vert > 82.5$. Hence,
the identity 
$$ \frac{\varphi_n'(0)}{n} =
Q_{j,n}(-744) = 
\left( Q_{\gamma_2,n}(x) + 2 x Q_{\gamma_2,n}'(x)+ \frac{x^2}{2}  Q_{\gamma_2,n}''(x)\right)_{\vert_{x= - 248}}
$$
restates and extends the result of \cite{AKN97}.

Now, let $g(n)$ be a normalized arithmetic function with moderate growth, 
such that $\sum_{n=1}^{\infty} \vert g(n) \vert \, T^n$
is analytic at $T=0$. Then the illustrated examples are 
special cases of polynomials $P_n^g(x)$ and $Q_n^g(x)$
defined by
\begin{eqnarray}
\sum_{n=0}^{\infty} P_n^g(z) \, q^n & = 
& \text{exp} \left( z \sum_{n=1}^{\infty} g(n) \frac{q^{n}}{n}\right), \label{P}
\\
\sum_{n=0}^{\infty} Q_n^g(z) \, q^n & = 
& \frac{1}{1- z \sum_{n=1}^{\infty^{\phantom{x}}} g(n) q^n}. \label{Q}
\end{eqnarray}
Note that $P_n^{\func{id}}(x)= x \, L_{n-1}^{(1)}(-x)$ are 
associated Laguerre polynomials (see \cite{HLN19}). Letting
$g(n)= \sigma(n)$, 
then we recover the polynomials provided by the Nekrasov--Okounkov hook length formula.
The polynomials $Q_n^{\func{id}}(x)$ are related to the Chebyshev polynomials of the second kind \cite{HNT20}.

It is easy to see that $P_n^g(z)$ and  $Q_n^g(z)$ are special cases of 
polynomials $P_n^{g,h}(x)$ defined by the recursion formula (\ref{gh}).
Here, $P_n^g(x) = P_n^{g,\func{id}}(x)$ and $Q_n^g(x) = P_n^{g,\mathbf{1}}(x)$.
In the next section, we state the main results of this paper.
\section{Statement of main results}

Let $g,h$ be arithmetic functions. Assume that $g$ be normalized
and $ 0 < h(n) \leq h(n+1)$.
It is convenient to extend $h$ by $h(0):=0$.

We start by recalling what is known \cite{HNT20, HN20A, HN20B}. 
Assume that $G_1(T):= \sum_{k=1}^{\infty} \vert g(k+1)\vert \, T^k$ has a positive radius $R$ of convergence.
Let $\kappa > 0$ be given, such that $G_1(2/\kappa) \leq \frac{1}{2}$. Let $x \in \mathbb{C}$.
Then we have for all $\vert x \vert  > \kappa \,\, h(n-1)$:
\begin{equation}\label{old}
    \frac{ \vert x \vert }{2 \, h(n) }   \vert P_{n-1}^{g,h}(x) \vert
    < \vert P_n^{g,h}(x) \vert    < \frac{ 3 \, \vert  x \vert }{2 \, h(n) }   \left\vert P_{n-1}^{g,h}(x)\right| .   
\end{equation}
This implies that $P_n^{g,h}(x)\neq 0$ for all $\vert x \vert  > \kappa \,\, h(n-1)$ and
$(-1)^n P_n^{g,h}(x) > 0$ if $x<-\kappa h\left( n-1\right) $.
Let $g(n)= \sigma(n)$. In \cite{HN20B} we proved
that $\kappa=15$ can is an acceptable value.
In the following we state our two main results: Improvement A and Improvement B. 
\subsection{Improvement A}
The following result
reproduces our previous result (\ref{old}), if we choose $\varepsilon = \frac{1}{2}$.
\begin{theorem}
\label{schwaecher}Let $0<\varepsilon <1$. Let
$R>0$ be the radius of convergence of
\[
G_{1}\left( T\right) =\sum _{k=1}^{\infty }\left| g\left( k+1\right) \right| T^{k}.
\]
Let $0<T_{\varepsilon }<R$ be such that
$G_{1}\left( T_{\varepsilon }\right) \leq \varepsilon $
and $\kappa=\kappa_{\varepsilon }=\frac{1}{1-\varepsilon }\frac{1}{T_{\varepsilon }}$.
Then
\begin{equation}\label{Th1:formel}
\left| P_{n}^{g,h}\left( x\right) - \frac{x}{h\left( n\right) } P_{n-1}^{g,h}\left( x\right) \right| < \varepsilon \frac{\left| x\right| }{h\left( n\right) }\left| P_{n-1}^{g,h}\left( x\right) \right|,
\end{equation}
if
$\left| x \right| >  {\kappa}\,\, h(n-1)$ for all
$n\geq 1$.
\end{theorem}
This result can be reformulated in
the following way, which is more suitable for applications
to growth and non-vanishing properties.
\begin{theorem}\label{versionA2}
Let $0<\varepsilon <1$. Let
$R>0$ be the radius of convergence of
\[
G_{1}\left( T\right) =\sum _{k=1}^{\infty }\left| g\left( k+1\right) \right| T^{k}.
\]
Let $0<T_{\varepsilon }<R$ be such that
$G_{1}\left( T_{\varepsilon }\right) \leq \varepsilon $
and $\kappa=\kappa_{\varepsilon }=\frac{1}{1-\varepsilon }\frac{1}{T_{\varepsilon }}$.
Then
\begin{equation}
\left( 1-\varepsilon \right) \frac{\left| x\right|
}{h\left( n\right) }\left|P_{n-1}^{g,h}\left( x\right) \right| <\left| P_{n}^{g,h}\left( x\right) \right| < \left( 1+\varepsilon \right) \frac{\left| x\right|
}{h\left( n\right) }\left| P_{n-1}^{g,h}\left( x\right) \right|,
\end{equation}
if
$\left| x \right| >  {\kappa}  \,\, h(n-1)
$ for all
$n\geq 1$.
\end{theorem}

\begin{corollary}
\label{schwaechererkorollar} Let $\kappa$ be chosen as in Theorem
\ref{schwaecher} or as in Theorem \ref{versionA2}. Then
$$P_{n}^{g,h}\left( x\right) \neq 0 \text{ for }
\left| x\right| >\kappa \,\, h(n-1).$$
\end{corollary}
\begin{proof}
This follows from Theorem \ref{versionA2}, since 
$\left( 1-\varepsilon \right) \frac{\left| x\right|
}{h\left( n\right) } \neq 0$ and $P_0^{g,h}(x)=1$.
\end{proof}
We note that the smallest possible $\kappa$ is independent of the function $h(n)$.
It is also possible to provide a lower bound for the best possible $\kappa$.

\begin{proposition}
The constant $\kappa_{\varepsilon }$ obtained in
Theorem \ref{schwaecher} has the following lower bound:
\[
\kappa_{\varepsilon }\geq \frac{\left| g\left( 2\right) \right| }{\left( 1-\varepsilon \right) \varepsilon }.
\]
As a lower bound\/ \emph{independent} of $\varepsilon $ we
have $4\left| g\left( 2\right) \right| $.
\end{proposition}

\begin{proof}
If we consider only the first
order term of the power series
\[
G_{1}\left( T\right) =\sum _{k=1}^{\infty }\left| g\left( k+1\right) \right| T^{k},
\]
then for positive $T$ we always have
$G_{1}\left( T\right) =\sum _{k=1}^{\infty }\left| g\left( k+1\right) \right| T^{k}\geq \left| g\left( 2\right) \right| T$.
Thus, $G_{1}\left( T\right) >\varepsilon $ if
$T>\frac{\varepsilon }{\left| g\left( 2\right) \right| }$. 
The case 
$G_{1}\left( T\right) \leq \varepsilon $ is only possible if
$T\leq \frac{\varepsilon }{\left| g\left( 2\right) \right| }$.
This forces
$T_{\varepsilon }\leq \frac{\varepsilon }{\left| g\left( 2\right) \right| }$.

Applying the last inequality now to
\begin{equation}
\kappa_{\varepsilon }:=\frac{1}{\left( 1-\varepsilon \right) T_{\varepsilon }}
\label{eq:kappaschwaecher}
\end{equation}
Theorem~\ref{schwaecher} shows the lower bound
$\kappa_{\varepsilon }\geq \frac{\left| g\left( 2\right) \right| }{\left( 1-\varepsilon \right) \varepsilon }$
in the
proposition depending on $\varepsilon $. The minimal value of this lower bound is
at $\varepsilon =\frac{1}{2}$ because of the inequality of
arithmetic and geometric means
$\left( 1-\varepsilon \right) \varepsilon \leq \left( \frac{1-\varepsilon +\varepsilon }{2}\right) ^{2}=\frac{1}{4}$.
\end{proof}

\subsection{Improvement B}

\begin{theorem}
\label{main}Let $0<\varepsilon <1$. Let
$R>0$ be the radius of convergence of
\[
G_{2}\left(  T \right) =\sum _{k=2}^{\infty }\left| g\left( k+1\right) -g\left( 2\right) g\left( k\right) \right| T^{k}.
\]
Let $0<T_{\varepsilon }<R$ be such that
$G_{2}\left( T_{\varepsilon }\right) \leq \varepsilon$ and  
$$ \kappa=\kappa_{\varepsilon }:=\frac{1}{1-\varepsilon }\left( \frac{1}{T_{\varepsilon }}+\left| g\left( 2\right) \right| \right). $$
Then
\begin{equation}\label{B1}
\left| P_{n}^{g,h}\left( x\right) - \frac{x+g\left( 2\right) h\left( n-1\right) }{h\left( n\right) } P_{n-1}^{g,h}\left( x\right) \right| < \varepsilon \frac{\left| x\right| }{h\left( n\right) }\left| P_{n-1}^{g,h}\left( x\right) \right|
\end{equation}
if
$\left| x \right| >  {\kappa} \,\, h(n-1)
$ for all $n\geq 1$.
\end{theorem}

\begin{theorem}
\label{variantB} Let $0 <\varepsilon <1$. Let
$R>0$ be the radius of convergence of 
\[
G_{2}\left(  T \right) =\sum _{k=2}^{\infty }\left| g\left( k+1\right) -g\left( 2\right) g\left( k\right) \right| T^{k}.
\]
Let $0<T_{\varepsilon }<R$ be such that
$G_{2}\left( T_{\varepsilon }\right) \leq \varepsilon$ and  
$$ \kappa=\kappa_{\varepsilon }:=\frac{1}{1-\varepsilon }\left( \frac{1}{T_{\varepsilon }}+\left| g\left( 2\right) \right| \right). $$
Then
\begin{eqnarray}
&&\frac{
\left| x+g\left( 2\right) h\left( n-1\right) \right| -\varepsilon \left| x\right|
}{h\left( n\right) }\left|P_{n-1}^{g,h}\left( x\right) \right| \nonumber \\
&<&\left| P_{n}^{g,h}\left( x\right) \right|  < \frac{
\left| x+g\left( 2\right) h\left( n-1\right) \right| +\varepsilon \left| x\right|
}{h\left( n\right) }\left| P_{n-1}^{g,h}\left( x\right) \right| \label{B2}
\end{eqnarray}
if
$\left| x \right| >  {\kappa} \,\, h (n-1)
$ for all
$n\geq 1$.
\end{theorem}

\begin{corollary}
\label{hauptkorollar} Let $\kappa$ be chosen as in Theorem
\ref{main} or as in Theorem \ref{variantB}. Then
\begin{equation} \label{B3} P_{n}^{g,h}\left( x\right) \neq 0 \text{ for }
\left| x\right| >\kappa \,\, h(n-1).
\end{equation}
\end{corollary}

\begin{proposition}
\label{untereschranke} The constant $\kappa_{\varepsilon }$ obtained in
Theorem \ref{main} has the following lower bound:
$$\kappa_{\varepsilon }\geq \frac{1}{1-\varepsilon }\left( \sqrt{\frac{\left| \left( g\left( 2\right) \right) ^{2}-g\left( 3\right) \right| }{\varepsilon }}+\left| g\left( 2\right) \right| \right).
$$
As a lower bound\/ \emph{independent} of $\varepsilon $ we
have $\frac{3}{2}\sqrt{3\left| \left( g\left( 2\right) \right) ^{2}-g\left( 3\right) \right| }+\left| g\left( 2\right) \right| $.
\end{proposition}

\begin{proof}
If we consider only the second order term of the power series
$G_{2}\left( T\right) =\sum _{k=2}^{\infty }\left| g\left( k+1\right) -g\left( 2\right) g\left( k\right) \right| T^{k}$,
then for positive $T$ we always have
\[
G_{2}\left( T\right) =\sum _{k=2}^{\infty }\left| g\left( k+1\right) -g\left( 2\right) g\left( k\right) \right| T^{k}\geq \left| \left( g\left( 2\right) \right) ^{2}-g\left( 3\right) \right| T^{2}.
\]
Thus, $G_{2}\left( T\right) >\varepsilon $ if
$T>\sqrt{\frac{\varepsilon }{\left| \left( g\left( 2\right) \right) ^{2}-g\left( 3\right) \right| }}$.
The case 
$G_{2}\left( T\right) \leq \varepsilon $ is only possible if
$T\leq \sqrt{\frac{\varepsilon }{\left| \left( g\left( 2\right) \right) ^{2}-g\left( 3\right) \right| }}$.
This forces
$T_{\varepsilon }\leq \sqrt{\frac{\varepsilon }{\left| \left( g\left( 2\right) \right) ^{2}-g\left( 3\right) \right| }}$.

Applying the last inequality now to
\begin{equation}
\kappa_{\varepsilon }:=\frac{1}{1-\varepsilon }\left( \frac{1}{T_{\varepsilon }}+\left| g\left( 2\right) \right| \right)
\label{eq:hauptkappa}
\end{equation}
from Theorem~\ref{main} shows the lower bound
$\kappa_{\varepsilon }\geq \frac{1}{1-\varepsilon }\left( \sqrt{\frac{\left| \left( g\left( 2\right) \right) ^{2}-g\left( 3\right) \right| }{\varepsilon }}+\left| g\left( 2\right) \right| \right) $
in the
proposition depending on $\varepsilon $.

It is clear that
\begin{equation}\frac{1}{1-\varepsilon }\left( \sqrt{\frac{\left| \left( g\left( 2\right) \right) ^{2}-g\left( 3\right) \right| }{\varepsilon }}+\left| g\left( 2\right) \right| \right) \geq \frac{1}{1-\varepsilon }\sqrt{\frac{\left| \left( g\left( 2\right) \right) ^{2}-g\left( 3\right) \right| }{\varepsilon }}+\left| g\left( 2\right) \right|
\label{eq:abschaetzung}
\end{equation}
for $0<\varepsilon <1$.
To estimate
$\kappa _{\varepsilon }$
independent of $\varepsilon $
we consider the right hand side of the last inequality as a
function in $\varepsilon $. Thus, we are interested in the
minimal value of this function for $0<\varepsilon <1$.
The inequality of arithmetic and geometric means yields
\begin{eqnarray*}
\left( 1-\varepsilon \right) \varepsilon ^{1/2}&=&2\left( \left( 1-\varepsilon \right) /2\right) ^{1/2}\cdot \left( \left( 1-\varepsilon \right) /2\right) ^{1/2}\cdot \varepsilon \\
&\leq &2\left( \frac{\left( 1-\varepsilon \right) /2+\left( 1-\varepsilon \right) /2+\varepsilon }{3}\right) ^{3/2}=\frac{2}{3\sqrt{3}}.
\end{eqnarray*}
We obtain
$\frac{3}{2}\sqrt{3\left| \left( g\left( 2\right) \right) ^{2}-g\left( 3\right) \right| }+\left| g\left( 2\right) \right| $.
\end{proof}
\subsection{Comparing Improvement A and Improvement B}
Let $0<\varepsilon _{1}<1$ and $T_{\varepsilon _{1}}$ as in
Theorem~\ref{schwaecher}. For all $T\geq 0$ we have that
\begin{eqnarray*}
G_{2}\left( T\right)
&\leq & \sum _{k=2}^{\infty }\left( \left| g\left( k+1\right) \right| +\left| g\left( 2\right) g\left( k\right) \right| \right) T^{k} \\
&=&
\left( 1+\left| g\left( 2\right) \right| T\right) G_{1}\left( T\right) -\left| g\left( 2\right) \right| T.
\end{eqnarray*}
Let
$\varepsilon _{2}$ be such that
\[
\left( 1+\left| g\left( 2\right) \right|
T_{\varepsilon _{1}}\right) G_{1}\left( T_{\varepsilon _{1}}\right) -\left| g\left( 2\right) \right| T_{\varepsilon _{1}}\leq \varepsilon _{2}\leq \left( 1+\left| g\left( 2\right) \right| T_{\varepsilon _{1}}\right) \varepsilon _{1}-\left| g\left( 2\right) \right| T_{\varepsilon _{1}}
<1.
\]
Then
\[
0\leq G_{2}\left( T_{\varepsilon _{1}}\right) \leq \left( 1+\left| g\left( 2\right) \right| T_{\varepsilon _{1}}\right) G_{1}\left( T_{\varepsilon _{1}}\right) -\left| g\left( 2\right) \right| T_{\varepsilon _{1}}
\leq
\varepsilon _{2}.
\]
This shows that we can choose
$T_{\varepsilon _{2}}=T_{\varepsilon _{1}}$.

Let $\kappa _{1,\varepsilon }$ and $\kappa _{2,\varepsilon }$
be the respective constants from Theorems~\ref{schwaecher}
and~\ref{main}. Then
\begin{eqnarray*}
\kappa _{2,\varepsilon _{2}}&=&\frac{1}{1-\varepsilon _{2}}\left( \frac{1}{T_{\varepsilon _{1}}}+\left| g\left( 2\right) \right| \right)
=
\frac{1}{1-\varepsilon _{2}}\left( 1+\left| g\left( 2\right) \right| T_{\varepsilon _{1}}\right) \frac{1}{T_{\varepsilon _{1}}} \\
&\leq
&\frac{1}{1-\left( 1+\left| g\left( 2\right) \right| T_{\varepsilon _{1}}\right) \varepsilon _{1}+\left| g\left( 2\right) \right| T_{\varepsilon _{1}}
}\left( 1+\left| g\left( 2\right) \right| T_{\varepsilon _{1}}\right) \frac{1}{T_{\varepsilon _{1}}} \\
&=&\frac{1}{1-\varepsilon _{1}}\frac{1}{T_{\varepsilon _{1}}}=\kappa _{1,\varepsilon _{1}}.
\end{eqnarray*}
This shows that the minimal value of the
$\kappa _{2,\varepsilon }$ is never larger than the
minimal value of the $\kappa _{1,\varepsilon }$.



\section{Applications}
\subsection{Toy example}
Let us consider the case $g(n)=1$ for all $n \in \mathbb{N}$.
We observe that $G_2(T)=0$
for all $T$. Let $0< \varepsilon < 1$. Then we apply Theorem \ref{variantB}.
For all $\vert x \vert > \frac{1}{1- \varepsilon} h(n-1)$ we obtain
\begin{equation*}
\frac{
\left| x+ h\left( n-1\right) \right| -\varepsilon \left| x\right|
}{h\left( n\right) }\left|P_{n-1}^{\mathbf{1},h}\left( x\right) \right| 
<\left| P_{n}^{\mathbf{1},h}\left( x\right) \right|  < \frac{
\left| x+ h\left( n-1\right) \right| +\varepsilon \left| x\right|
}{h\left( n\right) }\left| P_{n-1}^{\mathbf{1},h}\left( x\right) \right|.
\end{equation*}
Let $\varepsilon \rightarrow 0$, then for all $\vert x \vert  > h(n-1)$:
\begin{equation*}
\frac{
\left| x+ h\left( n-1\right) \right|
}{h\left( n\right) }\left|P_{n-1}^{\mathbf{1},h}\left( x\right) \right| 
\leq \left| P_{n}^{\mathbf{1},h}\left( x\right) \right|  \leq \frac{
\left| x+ h\left( n-1\right) \right| 
}{h\left( n\right) } \left| P_{n-1}^{\mathbf{1},h}\left( x\right) \right|.
\end{equation*}
Then,
$\left\vert P_{n}^{\mathbf{1},h}\left( x\right) \right| = \prod_{k=0}^{n-1}   \frac{\vert x + h(k)\vert}{h(k+1)}$ (we define $h(0):=0$).
Since $P_{1}^{\mathbf{1},h}\left( x\right) = x/h(1)$ and
$P_{n}^{\mathbf{1},h}\left( x\right) $ is a polynomial
of degree $n$ with positive leading coefficient, it follows:
\begin{equation}
P_n^{\mathbf{1},h}(x) = \frac{ x (x+h(1))
\cdots (x+h(n-1))}{h(1)
\cdots h(n)}.
\end{equation}
\subsection{Reciprocals of Eisenstein series}
Let $\sigma_{k}(n)= \sum_{d \vert n} d^{k}$ and let $B_{k}$ be the $k$th Bernoulli number. Then
we define Eisenstein series of weight $k$:
\begin{equation}
E_k(\tau):= 1 - \frac{2k}{B_k} \sum_{n=1}^{\infty} \sigma_{k-1}(n) \, q^n \qquad (k=2,4,6,\ldots).
\end{equation}
In \cite{AKN97} it was indicated that the $q$-expansion of the reciprocal of $E_4(\tau)=1 +240 \sum_{n=1}^{\infty} \sigma_3(n) \, q^n$
given by
\begin{equation}
\frac{1}{E_4(\tau)} = \sum_{n=0}^{\infty} \beta_n \, q^n,
\end{equation}
has strictly alternating sign changes: $(-1)^n \beta_n >0$.
Let $\varepsilon_1= \frac{1}{25}
$ and $\varepsilon_2= \frac{1}{982}
$. We can chose $\kappa$ in Theorem 1--4, such that $240 > \kappa$.
(In both cases $T_{\varepsilon }=\frac{87}{20000}$ does the job.
Then $\kappa _{1}=\frac{62500}{261}\approx 239.46$
and $\kappa _{2}=\frac{20408906}{85347}\approx 239.13$.
Note that an approximation of the smallest possible value that
can be obtained by our method is
$\kappa _{2}=\frac{539}{16}\approx 33.7$. This we obtain for
$\varepsilon _{2}=\frac{5}{21}$ and
$T_{\varepsilon _{2}}=\frac{3}{20}$.)

\begin{proof}[Proof of $\kappa _{2}\leq \frac{20408906}{85347}$]
Let $T_{\varepsilon }=\frac{87}{20000}$. Let further $\varepsilon _{1}=\frac{1}{25}$
and $\varepsilon _{2}=\frac{1}{982}$.
We have the well-known estimate
\begin{equation}
\sigma _{3}\left( k\right) 
\leq
\left( 1+\int _{1}^{
\infty }t^{-3}\,\mathrm{d}t\right) k^{3}=
3k^{3}/2
.
\label{eq:sigma3}
\end{equation}
Thus,
$\sigma _{3}\left( k\right) \leq 3k^{3}/2\leq 9\binom{k+2}{3}$.
Let $c_{1}\left( k\right) =\sigma _{3}\left( k+1\right) $
for $k\leq 2$ and
$c_{1}\left( k\right) =9\binom{k+3}{3}$ for $k\geq 3$.
Then
$G_{1}\left( T\right) \leq \sum _{k=1}^{\infty }c_{1}\left( k\right) T^{k}=9\frac{1}{\left( 1-T\right) ^{4}}-9-27T-62T^{2}$
and
$$G_{1}\left( \frac{87}{20000}\right) \leq \frac{1248274072444709335238721}{31446822595409952200000000}<\frac{1}{25}.$$
Thus,
$\kappa _{1}\leq \frac{20000}{87}\frac{25}{24}=\frac{62500}{261}\approx 239.46$.

With (\ref{eq:sigma3}) it also follows that
$
\left| 9\sigma _{3}\left( k\right) -\sigma _{3}\left( k+1\right) \right| \leq 15\left( k+1\right) ^{3}\leq 90\binom{k+3}{3}$.
Let $c_{2}\left( k\right) =\left| 9\sigma _{3}\left( k\right) -\sigma _{3}\left( k+1\right) \right| $
for $k\leq 4$ and $c_{2}\left( k\right) =90\binom{k+3}{3}$
for $k\geq 5$. Then
$G_{2}\left( T\right) \leq \sum _{k=2}^{\infty }c_{2}\left( k\right) T^{k}=\frac{90}{\left( 1-T\right) ^{4}}-90-360T-847T^{2}-1621T^{3}-2619T^{4}$
for $T>0$ and $$G_{2}\left( \frac{87}{20000}\right) \leq 
\frac{25605878110865247894531439480101}{25157458076327961760000000000000000}<\frac{1}{982}.$$
Thus,
$\kappa _{2}\leq \left( \frac{20000}{87}+9\right)
\frac{982}{981}=\frac{20408906}{85347} \approx 239.13$.
\end{proof}
\ \newline \

Note that $\beta_1= -240$, $\beta_n \in \mathbb{Z}$ and $\beta_1 \mid \beta_n$ for all $n \geq 1$. From
(\ref{old}), Theorem 1--4 and Corollary \ref{schwaechererkorollar} the following properties are obtained.
\begin{eqnarray*}
\frac{1}{2} \vert \beta_1 \beta_{n-1} \vert < & \vert \beta_n \vert & < \frac{3}{2} \vert \beta_1 \beta_{n-1} \vert,                       \\
\vert \beta_n -  \beta_1 \vert \, \left| \beta_{n-1} \right\vert &  < & \varepsilon_1 \, \vert \beta_1  \beta_{n-1} \vert,\\
(1 - \varepsilon_1) \vert \beta_1 \beta_{n-1} \vert  < & \vert \beta_n \vert & < (1 + \varepsilon_1) \vert \beta_1 \beta_{n-1} \vert, \\
 \vert \beta_n -( \beta_1 + 9) \vert & < & \varepsilon_2 \vert \beta_1 \beta_{n-1} \vert,\\
 \vert 231 + \varepsilon_2 \beta_1 \vert \,  \vert  \beta_{n-1} \vert < & \vert \beta_n \vert & < 
\vert 231 - \varepsilon_2 \beta_1  \vert \, \vert \beta_{n-1}\vert.
\end{eqnarray*}
Since $\beta_0 =1$ we can deduce that $(-1)^n \beta_n >0$.

In the previous proof we showed that
$G_{2}\left( T_{\varepsilon }\right) <\frac{1}{982}<\frac{1}{250}$
for $T_{\varepsilon }=\frac{87}{20000}$
and $\kappa _{2}<240$.
This leads to the following
\begin{theorem}[\cite{HN20C}]
\label{abschaetzung}Let
$G_{2}\left( T
\right)$ be defined by
$$ \sum _{m=2}^{\infty }\left| \sigma_{3} \left( m+1\right) - 9
\sigma_{3} \left( m\right) \right| T^m $$
with positive radius  of convergence $R$. Suppose that there is
$0<T_{\varepsilon }<1$ such that
$G_{2}\left( T_{\varepsilon }\right) \leq \frac{1}{250} $ and
$\kappa_{2 }  \leq \frac{250}{249}\left( \frac{1}{T_{\varepsilon }}+ \sigma_{3} \left( 2\right) \right) <\frac{8}{\left| B_{4}\right| }=240$,
then
the absolute value of the
$n$th coefficient $\beta_n$ of $1/E_{4}$
can be estimated by
\begin{equation}
240
\left( \left( 1\pm \frac{1}{250}
\frac{240}{231}
\right)
231
\right) ^{n-1}. 
\end{equation}
\end{theorem}

This implies
\begin{equation}
230^{n-1}\leq \frac{\left( -1\right) ^{n}\beta _{n}}{240}\leq 232^{n-1}.
\label{eq:beta}
\end{equation}
The following table displays the first values.
\begin{center}
\begin{minipage}[t]{1.0\textwidth}
{\small
\[
\begin{array}{|r||r|r|r|}
\hline
n & 230^{n-1} & \frac{\beta _{n}}{240} & 232^{n-1} \\ \hline \hline
1 & 1 & -1 & 1 \\ \hline
2 & 230 & 231 & 232 \\ \hline
3 & 52900 & -53308 & 53824 \\ \hline
4 & 12167000 & 12301607 & 12487168 \\ \hline
5 & 2798410000 & -2838775326 & 2897022976 \\ \hline
6 & 643634300000 & 655088819748 & 672109330432 \\ \hline
7 & 148035889000000 & -151171301803544 & 155929364660224 \\ \hline
8 & 34048254470000000 & 34884983226375975 & 36175612601171968 \\ \hline
9 & 7831098528100000000 & -8050218792755033557 & 8392742123471896576 \\ \hline
10 & 1801152661463000000000 & 1857705425589167301906 & 1947116172645480005632 \\
 \hline
\end{array}
\]}
\captionsetup{margin={0cm,0cm,0cm,0cm}}
\captionof{table}{Estimation given by (\ref{eq:beta})
}
\end{minipage}
\end{center}
By dividing $\beta _{n}$ by the estimates we obtain
the figures displayed in Table \ref{norm}:
\begin{center}
\begin{minipage}[t]{1.0\textwidth}
\[
\begin{array}{|r||r|r|}
\hline
n & \frac{\beta _{n}}{240\cdot 230^{n-1}} & \frac{\beta _{n}}{240\cdot 232^{n-1}} \\ \hline \hline
0 & -1.00000000 & -1.00000000 \\ \hline
1 & 1.00434783 & 0.99568966 \\ \hline
2 & -1.00771267 & -0.99041320 \\ \hline
3 & 1.01106329 & 0.98513987 \\ \hline
4 & -1.01442438 & -0.97989396 \\ \hline
5 & 1.01779663 & 0.97467598 \\ \hline
6 & -1.02118009 & -0.96948578 \\ \hline
7 & 1.02457479 & 0.96432322 \\ \hline
8 & -1.02798078 & -0.95918815 \\ \hline
9 & 1.03139810 & 0.95408043 \\ \hline
\end{array}
\]
\captionsetup{margin={0cm,0cm,0cm,0cm}}
\captionof{table}{Normalization} \label{norm}
\end{minipage}
\end{center}

\begin{remark}
The value $\varepsilon _{2}=\frac{1}{982}$
improves the inequalities (\ref{eq:beta}) to
\[
230.7648^{n-1}\leq \frac{\left( -1\right) ^{n}\beta _{n}}{240}\leq 231.2353^{n-1}.
\]
The lower bound is quite close to the
optimal value $e^{\pi \sqrt{3}}
=230.764588\ldots $.
\end{remark}

\subsection{Associated Laguerre polynomials and Chebyshev polynomials of the second kind}
We briefly recall the definition of associated Laguerre polynomials $L_n^{(\alpha)}(x)$ 
and Chebyshev polynomials $U_n(x)$ of the second kind \cite{RS02, Do16}. Both are orthogonal polynomials.
We have
\begin{equation}
L_{n}^{\left( \alpha \right) }\left( x\right) =
\sum _{k=0}^{n}\binom{n+\alpha }{n-k}\frac{(-x)^{k}}{k!} \qquad (\alpha > -1).
\end{equation}
The Chebyshev polynomials are uniquely characterized by
\begin{equation}
U_{n}(\text{cos}(t)) = \frac{\text{sin}((n+1)t)}{\text{sin}(t)} \qquad ( 0 < t < \pi).
\end{equation}
The Chebyshev polynomials are of special interest in the context of applications, 
since they are the only classical orthogonal polynomials whose
zeros can be determined in explicit form (see Rahman and Schmeisser \cite{RS02}, Introduction).
Let $g(n)=\func{id}(n)= n$. Then
\begin{eqnarray}
P_n^{\func{id}}(x) & =& \frac{x}{n} L_{n-1}^{(1)}(-x),\\
Q_n^{\func{id}}(x) & =& x \, U_{n-1}\left(\frac{x}{2} + 1 \right).
\label{eq:tschebyscheff}
\end{eqnarray}
The generating series of the Chebyshev polynomial of the second kind is given by
\begin{equation}
\sum_{n=0}^{\infty} U_n(x) \, q^n = \frac{1}{1 - 2x q + q^2}, \qquad \vert x \vert, \vert q \vert <1.
\end{equation}
With this we can
prove equation (\ref{eq:tschebyscheff}). We have
\begin{eqnarray*}
1+xq\sum _{n=0}^{\infty }U_{n}\left( \frac{x}{2}+1\right) q^{n}&=&
1+\frac{xq}{1-\left( 2+x\right) q+q^{2}}=
\frac{1-2q+q^{2}}{1-\left( 2+x\right) q+q^{2}}\\
&=&\frac{1}{1-xq\frac{1}{\left( 1-q\right) ^{2}}}
=
\frac{1}{1-xq\sum _{n=1}^{\infty }nq^{n-1}}\\
&=&\sum _{n=0}^{\infty }Q_{n}\left( x\right) q^{n}
\end{eqnarray*}
using Definition (\ref{Q}).
Note that
$G_{1}\left( T\right) =\sum _{k=1}^{\infty }\left( k+1\right) T^{k}=\frac{1}{\left( 1-T\right) ^{2}}-1$
and
\[
G_{2}\left( T\right) =\sum _{k=2}^{\infty }\left( k-1\right) T^{k}=\frac{T^{2}}{\left( 1-T\right) ^{2}}.
\]
From this
we obtain the
following values:
\begin{center}
\begin{minipage}[t]{0.6\textwidth}
\[
\begin{array}{|c|c|c||c|c|c|}
\hline
\varepsilon _{1}&T_{\varepsilon _{1}}&\kappa _{1}&\varepsilon _{2}&T_{\varepsilon _{2}}&\kappa _{2} \\ \hline
\vphantom{\int }\frac{11}{25}&\frac{1}{6}&\frac{75}{7}&\frac{1
}{4
}&\frac{1
}{3
}&\frac{20
}{3
} \\ \hline
\end{array}
\]
\captionsetup{margin={0cm,0cm,0cm,0cm}}
\captionof{table}{Case $g(n)=n$}
\end{minipage}
\end{center}
If we consider the special case $\varepsilon_1 = 1/2$ in Improvement A, we can chose $T_{\varepsilon _{1}}=2/11$ and finally get $\kappa_1 = 11$.

This
leads to several applications. For example, let $\vert x \vert >\left( 20/3\right) \, n $ then $L_n^{(1)}(x) \neq 0$ and the estimates hold
\begin{equation*}
\left|(\vert x + 2n \vert - 1/4 \vert x \vert )\right| \,\, \vert L_{n-1}^{(1)}(x)\vert < n  \vert L_{n}^{(1)}(x)\vert < 
|(\vert x + 2n \vert + 1/4 \vert x \vert) | \,\, \vert L_{n-1}^{(1)}(x)\vert.
\end{equation*}

\subsection{Powers of the Dedekind $\eta$-function.}
Let us recall the well-known identity:
\begin{equation}
\prod_{n=1}^{\infty} \left( 1 - q^n \right) = \text{exp} \left( - \sum_{n=1}^{\infty} \sigma(n) \, \frac{q^n}{n} \right) \qquad (z \in \mathbb{C}).
\end{equation}
The $q$-expansion of the $-z$th power of the Euler product defines the
D'Arcais polynomials 
\begin{equation}
\sum_{n=0}^{\infty} P_n^{\sigma}(z) \, q^n = \prod_{n=1}^{\infty} \left( 1 - q^n \right)^{-z} \qquad ( z \in \mathbb{C}),
\end{equation}
where $P_0^{\sigma}(x)=1$ and
$P_n^{\sigma}\left( x\right) = \frac{x}{n} \sum_{k=1}^{n} \sigma(k) P_{n-k}^{\sigma}(x)$, as polynomials.
Note that these polynomials evaluated at $-24$ are directly related to the Ramanujan $\tau$-function: $\tau(n)= P_{n-1}^{\sigma}(-24)$,
which gives also a link to the Lehmer conjecture \cite{Le47}.

In the spirit of this paper, let $\varepsilon :=\frac{3}{14}$. Then $T_{\varepsilon }: =\frac{2}{11}$
satisfies the assumptions of Theorem \ref{variantB}. We obtain the
\begin{corollary}\label{top}
Let $\kappa = \frac{119}{11}$. Then $P_n^{\sigma}(z) \neq 0$ for all complex $z$ with $\vert z \vert  > \kappa \,\, (n-1)$.
\end{corollary}

We have to show that
$G_2\left( T_{\varepsilon }\right) =
\sum _{k=2}^{\infty }\left| \sigma \left( k+1\right) -3\sigma \left( k\right) \right| T_{\varepsilon }^{k}<\varepsilon $. 
For this let 
$c\left( k\right) =\left| \sigma \left( k+1\right) -3\sigma \left( k\right) \right| $
for $1\leq k\leq 7$ and
$c\left( k\right) =4\binom{
k+2}{2}
$ for $k\geq 8$. Then
$\left| \sigma \left( k+1\right) -3\sigma \left( k\right) \right| \leq c\left( k\right) $
for all $k\in \mathbb{N}$ since
$$\sigma \left( k\right) \leq \left( 1+\ln \left( k\right) \right) k\leq \left( \frac{k}{4}+\ln \left( 4\right) \right) k\leq \binom{k+1}{2}$$
for $k\geq 4$.
This
implies
$G_2\left( T \right) \leq \sum _{k=2}^{\infty }c\left( k\right) T^{k}$
for $0\leq T \leq 1\leq R$. The upper bound is
now almost, except for the first $8$ terms, a multiple of the second derivative of the geometric
series of $T$. Hence,
\begin{equation*} G_2(T) \leq
\frac{4}{\left( 1-T\right) ^{3}}-4
-12T-
19T^{2}-
35T^{3}-45T^{4}-78T^{5}-84T^{6}-135T^{7}.
\end{equation*}
For $T=T_{\varepsilon }=\frac{2}{11}$ we obtain
\begin{equation*}
G_2 \left( T_{\varepsilon }\right) \leq    \frac{3043993780}{14206147659}   <   \frac{3}{14}=\varepsilon .
\end{equation*}
The claim now follows from Corollary \ref{hauptkorollar}.
\begin{remark}
\begin{itemize} 
\item[]
\item[a)]Let $\varepsilon$ and $\kappa$ be as above, and let 
$h$ be an arbitrary arithmetic function with $0 < h(n) \leq h(n+1)$. 
Then $P_n^{\sigma,h}(x)$ satisfies (\ref{B1}), (\ref{B2}), and (\ref{B3}) obtained by Improvement B.
\item[b)] The value $\varepsilon = \frac{3}{14}$ already leads to 
$$\kappa_{\varepsilon} = \frac{119}{11} = 10.\overline{81}.$$ Note only minor further improvements can be achieved. 
\item[c)] Corollary \ref{top} improves our previous result \cite{HN20B}, where $\kappa =15$.
\end{itemize}
\end{remark}

\begin{proposition}
Let $\varepsilon = 0.217$ and $T_{\varepsilon}= 0.18289$. Then the assumptions of Theorem \ref{main}
are fulfilled. Furthermore we can take $\kappa=10.815$.
\end{proposition}
\begin{proof}
Let $\varepsilon$ and $T_{\varepsilon}$ be given.
We have to show that
$$G_2\left( T_{\varepsilon }\right) =\sum _{k=2}^{\infty }\left| \sigma \left( k+1\right) -
3\sigma \left( k\right) \right| T_{\varepsilon }^{k}<\varepsilon . $$
Let $c\left( k\right) =\left| \sigma \left( k+1\right) -3\sigma \left( k\right) \right| $
for $1\leq k\leq 11$ and
$c\left( k\right) =4\binom{
k+2}{2}
$ for $k\geq 12$. Then
$\left| \sigma \left( k+1\right) -3\sigma \left( k\right) \right| \leq c\left( k\right) $
for all $k\in \mathbb{N}$ as
\[
\sigma \left( k\right) \leq \left( 1+\ln \left( k\right) \right) k\leq \left( \frac{k}{4}+\ln \left( 4\right) \right) k\leq \binom{k+1}{2}
\]
for $k\geq 4$.
This
implies
$G_2\left( q\right) \leq \sum _{k=2}^{\infty }c\left( k\right) T^{k}$
for $0\leq T \leq 1\leq R$. The upper bound is almost
(except for the first $12$ terms) a multiple
of the second derivative of the geometric
series of $T$. Hence $G_2(T) \leq \sum _{k=2}^{\infty }c\left( k\right) T^{k} \leq$
\begin{eqnarray*}
&&4\sum _{k=0}^{\infty }\binom{k+2}{2} T^{k}-4-12T-19T^{2}-35T^{3}-45T^{4}-78T^{5}-84T^{6}-135T^{7} \\
&&{}-148T^{8}-199T^{9}-222T^{10}-304T^{11} \\
&=&
\frac{4}{\left( 1-T\right) ^{3}}-4
-12T-
19T^{2}-
35T^{3}-45T^{4}-78T^{5}-84T^{6}-135T^{7} \\
&&{}-148T^{8}-199T^{9}-222T^{10}-304T^{11}.
\end{eqnarray*}
For $T=T_{\varepsilon }=0.18289$ we obtain
\begin{equation*}
G_2
\left( T_{\varepsilon }\right)
 < 0.216998<\varepsilon .
\end{equation*}
The claim now follows from
Corollary~\ref{hauptkorollar}.
\end{proof}

\begin{minipage}[t]{0.99\textwidth}
\[
\begin{array}{|c||c|c|c|c|c|c|c|c|c|c|c|c|c|c|}
\hline
k & 1 &
2 &
3 &
4 &
5 &
6 &
7 &
8 &
9 &
10 &
11 &
12 &
13 &
14 \\ \hline
\left| \sigma \left( k+1\right) -3\sigma \left( k\right) \right| & 0 &
5 &
5 &
15 &
6 &
28 &
9 &
32 &
21 &
42 &
8 &
70 &
18 &
48 \\ \hline
\end{array}
\]
\captionsetup{margin={0cm,0cm,0cm,0cm}}
\captionof{table}{Values of $\left| \sigma \left( k+1\right) -3\sigma \left( k\right) \right| $}
\end{minipage}
\section{Proof of Theorem \ref{schwaecher} and Theorem \ref{versionA2}}
\begin{proof}[Proof of Theorem~\ref{schwaecher}]
The proof will be by induction on $n$.
The case $n=1$ is obvious:
$\left| P_{1}^{g,h}\left( x\right) - \frac{x}{h\left( 1\right) }P_{0}^{g,h}\left( x\right) \right| =0<\varepsilon \frac{\left| x\right| }{h\left( 1\right) }\left| P_{0}^{g,h}\left( x\right) \right| $
for
$\left| x\right| >\kappa \,\, h( 0)$.

Let now $n\geq 2$. Then
\[
P_{n}^{g,h}\left( x\right)
=
\frac{x}{h\left( n\right) }\left( P_{n-1}^{g,h}\left( x\right) +\sum _{k=1}^{n-1}g\left( k+1\right) P_{n-1-k}^{g,h}\left( x\right) \right) .
\]
The basic idea for the induction step is to use the
inequality
\[
\left| P_{n}^{g,h}\left( x\right) -\frac{x}{h\left( n\right) }P_{n-1}^{g,h}\left( x\right) \right|
\leq
\frac{\left| x\right| }{h\left( n\right) }\sum _{k=1}^{n-1}\left| g 
\left( k+1\right) \right| \left| P_{n-1-k}^{g,h}\left( x\right) \right| .
\]
We estimate the sum by the following property for $1\leq j\leq n-1$:
\begin{eqnarray*}
\left| P_{j}^{g,h}\left( x\right) \right|
&\geq &
\left| \frac{x}{h\left( j\right) }\right| \left| P_{j-1}^{g,h}\left( x\right) \right| -\left| P_{j}^{g,h}-\frac{x}{h\left( j\right) }P_{j-1}^{g,h}\left( x\right) \right| \\
&> &\left( \frac{\left| x\right| }{h\left( j\right) }-\varepsilon \frac{\left| x\right| }{h\left( j\right) }\right) \left| P_{j-1}^{g,h}\left( x\right) \right| \\
&=&\frac{\left( 1-\varepsilon \right) \left| x\right| }{h\left( j\right) }\left| P_{j-1}^{g,h}\left( x\right) \right| \end{eqnarray*}
for
$\left| x\right| > \kappa \,\, h(n-1)
$.
Thus,
\[
\left| P_{j-1}^{g,h}\left( x\right) \right|
<\frac{h\left( j\right) }{\left( 1-\varepsilon \right) \left| x\right| }\left| P_{j}^{g,h}\left( x\right) \right|.
\]
Further, we have
\begin{eqnarray*}
\left| P_{n-k}^{g,h}\left( x\right) \right|
&<&\left| P_{n-k+1}^{g,h}\left( x\right) \right| \frac{h\left( n-k+1\right)
}{\left( 1-\varepsilon \right) \left| x\right| }<\ldots \\
&<&\left| P_{n-1}^{g,h}\left( x\right) \right| \prod _{j=1}^{k-1}\frac{h\left( n-j\right)
}{\left( 1-\varepsilon \right) \left| x\right| } \\
&\leq &\left| P_{n-1}^{g,h}\left( x\right) \right| \left( \frac{h\left( n-1\right) }{\left( 1-\varepsilon \right) \left| x \right| }\right) ^{k-1}
\end{eqnarray*}
for
$\left| x\right| >\kappa \,\, h (n-1) \geq \kappa \,\, h( n-k)$
for all $2 \leq k \leq n$ by assumption.
Using this, we can now estimate the sum by
\[
\sum _{k=1}^{n-1}\left| g 
\left( k+1\right) \right| \left| P_{n-1-k}^{g,h}\left( x\right) \right|
<
\left| P_{n-1}^{g,h}\left( x\right) \right| \sum _{k=2}^{n-1}\left| g 
\left( k+1\right) \right| \left( \frac{h\left( n-1\right)
}{\left( 1-\varepsilon \right) \left| x\right| }\right) ^{k}
\]
and we obtain
\begin{eqnarray*}
&&\left| P_{n}^{g,h}\left( x\right) -\frac{x}{h\left( n\right) }P_{n-1}^{g,h}\left( x\right) \right| \\
&<&
\frac{\left| x\right| }{h\left( n\right) }\left| P_{n-1}^{g,h}\left( x\right) \right| \sum _{k=1}^{n-1}\left| g\left( k+1\right) \right| \left( \frac{h\left( n-1\right)
}{\left( 1-\varepsilon \right) \left| x\right| }\right) ^{k}.
\end{eqnarray*}
Estimating the sum using the
assumption from the theorem we obtain
\[
\sum_{k=1}^{n-1}
\left| g(k+1) \right| \left( \frac{h(n-1)}{\left( 1-\varepsilon \right) 
\left| x \right|} \right)^{k} 
\leq
G_{1}
\left( \frac{h\left( n-1\right)
}{\left( 1-\varepsilon \right) \left| x\right| }\right)
\leq G_{1}\left( T_{\varepsilon }\right) \leq
\varepsilon ,
\]
since
$\left| x\right| >\kappa \,\, h(n-1) =\frac{h\left( n-1\right) }{1-\varepsilon }\frac{1}{T_{\varepsilon }}$
which is equivalent to
$\frac{\left( 1-\varepsilon \right) \left| x\right| }{h\left( n-1\right) }>\frac{1}{T_{\varepsilon }}$
and $G_{1}$ increases on $\left[ 0,R\right) $
as $\left| g\left( k+1\right) \right| \geq 0$ for all $k\in \mathbb{N}$.
\end{proof}

\begin{proof}[Proof of Theorem \ref{versionA2}]
Consider the following upper and lower bounds:
\begin{eqnarray*}
\vert P^{g,h}_{n}\left( x\right) \vert & \leq &
\left|\frac{x}{h\left( n\right) }P^{g,h}_{n-1}\left( x\right) \right| 
+ \left| P^{g,h}_{n}\left( x\right) -\frac{x}{h\left( n
\right) }P^{g,h}_{n-1}\left( x\right) \right|, \\
\vert P^{g,h}_{n}\left( x\right) \vert & \geq &
\left|\frac{x}{h\left( n\right) }P^{g,h}_{n-1}\left( x\right) \right| 
- \left| P^{g,h}_{n}\left( x\right) -\frac{x}{h\left( n
\right) }P^{g,h}_{n-1}\left( x\right) \right|  .
\end{eqnarray*}
Applying (\ref{Th1:formel}) leads to the desired result.
\end{proof}

\section{Proof of Theorem \ref{main} and Theorem \ref{variantB}}
\begin{proof}[Proof of Theorem~\ref{main}]
The proof will be by induction on $n$.
The case $n=1$ is obvious:
$$\left| P_{1}^{g,h}\left( x\right) - \frac{x+g\left( 2\right) h\left( 0\right) }{h\left( 1\right) }P_{0}^{g,h}\left( x\right) \right| =0<\varepsilon \frac{\left| x\right| }{h\left( 1\right) }\left| P_{0}^{g,h}\left( x\right) \right| $$
for
$\left| x\right| >\kappa \,\, h(0) $. Let now $n\geq 2$. Then
\begin{eqnarray*}
&&P_{n}^{g,h}\left( x\right) -g\left( 2\right) \frac{h\left( n-1\right) }{h\left( n\right) }P_{n-1}^{g,h}\left( x\right) \\
&=&\frac{x}{h\left( n\right) }\left( P_{n-1}^{g,h}\left( x\right) +\sum _{k=2}^{n-1}\left( g\left( k+1\right) -g\left( 2\right) g\left( k\right) \right) P_{n-1-k}^{g,h}\left( x\right) \right) .
\end{eqnarray*}
The basic idea for the induction step is to use the
inequality
\begin{eqnarray*}
&&\left| P_{n}^{g,h}\left( x\right) -\frac{x+g\left( 2\right) h\left( n-1\right) }{h\left( n\right) }P_{n-1}^{g,h}\left( x\right) \right| \\
&\leq &\frac{\left| x\right| }{h\left( n\right) }\sum _{k=2}^{n-1}\left| g 
\left( k+1\right) -g\left( 2\right) g\left( k\right) \right| \left| P_{n-1-k}^{g,h}\left( x\right) \right| .
\end{eqnarray*}
The sum can be estimated using for $1\leq j\leq n-1$ that
\begin{eqnarray*}
&&\left| P_{j}^{g,h}\left( x\right) \right| \\
&\geq &\left| \frac{x+g\left( 2\right) h\left( j-1\right) }{h\left( j\right) }\right| \left| P_{j-1}^{g,h}\left( x\right) \right| -\left| P_{j}^{g,h}-\frac{x+g\left( 2\right) h\left( j-1\right) }{h\left( j\right) }P_{j-1}^{g,h}\left( x\right) \right| \\
&> &\left( \frac{\left| x\right| }{h\left( j\right) }-\frac{\left| g\left( 2\right) \right| h\left( j-1\right) }{h\left( j\right) }-\varepsilon \frac{\left| x\right| }{h\left( j\right) }\right) \left| P_{j-1}^{g,h}\left( x\right) \right| \\
&=&\frac{\left( 1-\varepsilon \right) \left| x\right| -\left| g\left( 2\right) \right| h\left( j-1\right) }{h\left( j\right) }\left| P_{j-1}^{g,h}\left( x\right) \right| \\
&\geq &\frac{\left( 1-\varepsilon \right) \left| x\right| -\left| g\left( 2\right) \right| h\left( j\right) }{h\left( j\right) }\left| P_{j-1}^{g,h}\left( x\right) \right|
\end{eqnarray*}
for
$\left| x\right| > \kappa \,\, h(n-1)
$.
Note that for
$\left| x\right| >\kappa \,h(n-1) $
we have
\[
\left( 1-\varepsilon \right)  \left| x \right| -g\left( 2\right) h\left( j\right) >\left( \frac{1}{T_\varepsilon }+\left| g\left( 2\right) \right| \right) h\left( n-1\right) -g\left( 2\right) h\left( j\right) >0.
\]
Thus,
\[
\left| P_{j-1}^{g,h}\left( x\right) \right|
<\frac{h\left( j\right) }{\left( 1-\varepsilon \right) \left| x\right| -g\left( 2\right) h\left( j\right) }\left| P_{j}^{g,h}\left( x\right) \right|.
\]
We use this inequality and obtain
\begin{eqnarray*}
\left| P_{n-k}^{g,h}\left( x\right) \right|
&<&\left| P_{n-k+1}^{g,h}\left( x\right) \right| \frac{h\left( n-k+1\right)
}{\left( 1-\varepsilon \right) \left| x\right| -\left| g\left( 2\right) \right| h\left( n-k+1\right) }<\ldots \\
&<&\left| P_{n-1}^{g,h}\left( x\right) \right| \prod _{j=1}^{k-1}\frac{h\left( n-j\right)
}{\left( 1-\varepsilon \right) \left| x\right| -\left| g\left( 2\right) \right| h\left( n-j\right) } \\
&\leq &\left| P_{n-1}^{g,h}\left( x\right) \right| \left( \frac{h\left( n-1\right) }{\left( 1-\varepsilon \right) \left| x \right| -\left| g\left( 2\right) \right| h\left( n-1\right) }\right) ^{k-1}
\end{eqnarray*}
for
$\left| x\right| >\kappa \,\, h( n-1) \geq \kappa \,\, h(n-k)
$
for all $2 \leq k \leq n$ by assumption.
Using this, we can now estimate the sum by
\begin{eqnarray*}
&&
\sum _{k=2}^{n-1}\left| g 
\left( k+1\right) -g\left( 2\right) g\left( k\right) \right| \left| P_{n-1-k}^{g,h}\left( x\right) \right| \\
&<&
\left| P_{n-1}^{g,h}\left( x\right) \right| \sum _{k=2}^{n-1}\left| g 
\left( k+1\right) -g\left( 2\right) g\left( k\right) \right| \left( \frac{h\left( n-1\right)
}{\left( 1-\varepsilon \right) \left| x\right| -\left| g\left( 2\right) \right| h\left( n-1\right) }\right) ^{k}
\end{eqnarray*}
and we obtain
\begin{eqnarray*}
&&\left| P_{n}^{g,h}\left( x\right) -\frac{x+g\left( 2\right) h\left( n-1\right) }{h\left( n\right) }P_{n-1}^{g,h}\left( x\right) \right| \\
&<&
\frac{\left| x\right| }{h\left( n\right) }\left| P_{n-1}^{g,h}\left( x\right) \right| \sum _{k=2}^{n-1}\left| g\left( k+1\right) -g\left( 2\right) g 
\left( k\right) \right| \left( \frac{h\left( n-1\right)
}{\left( 1-\varepsilon \right) \left| x\right| -\left| g\left( 2\right) \right| h\left( n-1\right) }\right) ^{k}.
\end{eqnarray*}
Estimating the sum using the
assumption from the theorem we obtain
\begin{eqnarray*}
&&\sum _{k=2}^{n-1}\left| g\left( k+1\right) -g\left( 2\right) g
\left( k\right) \right| \left( \frac{h\left( n-1\right)
}{\left( 1-\varepsilon \right) \left| x\right| -\left| g\left( 2\right) \right| h\left( n-1\right) }\right) ^{k} \\
&\leq &
G_2
\left( \frac{h\left( n-1\right)
}{\left( 1-\varepsilon \right) \left| x\right| -\left| g\left( 2\right) \right| h\left( n-1\right) }\right)
\leq G_2\left( T_{\varepsilon }\right) \leq
\varepsilon
\end{eqnarray*}
since
$\left| x\right| >\kappa \,\, h( n-1) =\frac{\kappa \,\,h(n-1) }{1-\varepsilon }\left( \frac{1}{T_{\varepsilon }}+\left| g\left( 2\right) \right| \right) $
which is equivalent to
$\frac{\left( 1-\varepsilon \right) \left| x\right| }{h\left( n-1\right) }-\left| g\left( 2\right) \right| >\frac{1}{T_{\varepsilon }}$
and $G_2$ is increasing on $\left[ 0,R\right) $
as $\left| g\left( k+1\right) -g\left( 2\right) g\left( k\right) \right| \geq 0$ for all $k\in \mathbb{N}$.
\end{proof}

\begin{proof}[Proof of Theorem \ref{variantB}]
This basically follows from Theorem \ref{main} (see also the proof of Theorem \ref{versionA2}).
\end{proof}

\begin{center}
\begin{minipage}[t]{0.6\textwidth}
\[
\begin{array}{|r|r|}
\hline
n & \min \left\{ \func{Re}\left( x\right) :P_{n}^{\sigma ,\func{id}}\left( x\right) =0\right\} \\ \hline \hline
1 & 0 \\ \hline
2 & -3 \\ \hline
3 & -8 \\ \hline
4 & -14 \\ \hline
5 & -20.61187 \\ \hline
6 & -27.64001 \\ \hline
7 & -34.97153 \\ \hline
8 & -42.53511 \\ \hline
9 & -50.28267 \\ \hline
10 & -58.18014 \\ \hline
50 & -410.63656 \\ \hline
100 & -874.47135 \\ \hline
500 & -4687.67815 \\ \hline 
1000 & -9501.75903 \\ \hline 
\end{array}
\]
\captionsetup{margin={0cm,0cm,0cm,0cm}}
\captionof{table}{Minimal zeros of $P_{n}^{\sigma ,\func{id}}\left( x\right) $}\label{roots}
\end{minipage}
\end{center}


\begin{Acknowledgement}
To be entered later.
\end{Acknowledgement}


\end{document}